\documentclass[11pt]{article}

\usepackage[colorlinks=true,linkcolor=blue,citecolor=blue]{hyperref}
\usepackage{subfigure}
\usepackage{setspace}
\usepackage{titletoc}
\usepackage{titlesec}
\usepackage{algorithmicx}
\usepackage{subfigure}
\usepackage{subcaption}
\usepackage{mathtools}
\usepackage{emptypage}
\usepackage{amsfonts,amssymb}
\usepackage{subfigure}
\usepackage{enumerate}
\usepackage{graphicx}
\usepackage{epstopdf}
\usepackage{etoolbox}
\usepackage{titlesec}
\usepackage[titles,subfigure]{tocloft}
\usepackage{cite}
\usepackage{mathrsfs}
\usepackage{listings}
\usepackage{amsmath}
\usepackage{tikz-cd}
\usepackage{pdfpages}
\usepackage{amsthm}
\usepackage{amsthm}
\usepackage{geometry}
\usepackage{fancyhdr}
\usepackage{lastpage}

\newtheorem{theorem}{Theorem}[section]
\newtheorem{lemma}[theorem]{Lemma}

\numberwithin{equation}{section}
\theoremstyle{definition}
\newtheorem{definition}[theorem]{Definition}
\theoremstyle{remark}
\newtheorem{remark}[theorem]{Remark}

\begin{document}
\noindent
\title{Comparison Results for a class of Neumann Problems of the $p$-Laplace Equation on Riemannian Manifolds}
{\Large \textbf{Comparison Results for a class of Neumann Problems of the $p$-Laplace Equation on Riemannian Manifolds}}\vspace{1\baselineskip}\par\noindent
{\large \textbf{Wentao Liu and Anqiang Zhu} }
\vspace{1\baselineskip}\par
\noindent
{\large \textbf{Abstract}}:We consider Neumann boundary value problems for the $p$-Laplace equation on Riemannian manifolds with nonnegative Ricci curvature. 
Using spherical symmetrization under appropriate constraints, we derive Talenti-type comparison results in Lorentz spaces.
We further show that, in contrast to the Robin case, the Neumann setting admits weaker constraints, which yields stronger comparison principles.
\vspace{0.5\baselineskip}\par\noindent
{\large \textbf{Keyword}}:Talenti comparison;Poisson equation;Neumann conditions;$p$-Laplace equation
\section{Introduction}
\par\noindent
For Neumann boundary value problems, 
solutions are determined only up to additive constants. 
Consequently, any meaningful comparison principle requires an additional normalization condition. 
Unlike the Dirichlet case, where the boundary condition naturally fixes the solution, 
the Neumann setting presents an intrinsic ambiguity, making comparison results substantially more delicate.

The origin of this circle of ideas can be traced back to the celebrated work of Talenti \cite{talenti1976elliptic}, 
who established a sharp comparison principle for elliptic equations with homogeneous Dirichlet boundary conditions by means of Schwarz symmetrization. 
More precisely, let $\Omega\subset\mathbb{R}^n$ be a bounded domain with Lipschitz boundary and consider the Poisson problem
\begin{equation}
\begin{cases}
-\Delta u=f,
& \text{in }\Omega,
\\
u=0,
& \text{on }\partial\Omega.
\end{cases}
\end{equation}

Together with its symmetrized counterpart
\begin{equation}
\begin{cases}
-\Delta v=f^\sharp,
& \text{in }\Omega^\sharp,
\\
v=0,
& \text{on }\partial\Omega^\sharp,
\end{cases}
\end{equation}
where $\Omega^\sharp$ is a Euclidean ball satisfying
$|\Omega^\sharp|=|\Omega|$ and $f^\sharp$ denotes the Schwarz symmetrization of $f$.

Talenti proved that the rearrangements of $u$ and $v$ satisfy a pointwise comparison, yielding a powerful method for deriving sharp functional inequalities. 
Since then, symmetrization techniques have become a fundamental tool in the study of elliptic partial differential equations, 
with applications ranging from Sobolev inequalities to Faber--Krahn type estimates.
The situation is considerably better for Dirichlet boundary conditions.
Starting from the seminal work of Talenti, symmetrization techniques
have led to a rich theory of comparison principles, often yielding
pointwise estimates between a solution and its symmetric counterpart.
We refer to the surveys
\cite{kawohl2006rearrangements,kesavan2006symmetrization}
for a detailed overview of the subject.

The influence of these ideas extends far beyond the Euclidean setting.
Talenti-type comparison results have recently been established on
Riemannian manifolds with positive Ricci curvature by Colladay et al.
\cite{colladay2018comparison}, and to nonnegative Ricci curvature by Chen et al.\cite{chen2021talenti}.
Later it is generalized to
$\mathrm{RCD}(K,N)$ spaces by Mondino et al.
\cite{mondino2021talenti}.

The comparison theorem for Robin boundary value problems was first considered by Alvino et al. \cite{alvino2023talenti}. 
These results were generalized to Riemannian manifolds by Chen et al.\cite{chen2023comparison}. 
And also were then extended to the 
$p$
-Laplacian by Amato et al. \cite{amato2022comparison}. 
Recently, Chen et al. \cite{chen2025talenti} generalized the corresponding conclusions to the anisotropic 
$p$
-Laplacian and proved rigidity properties for Talenti comparison inequalities.

Regarding Neumann boundary conditions, Maderna et al. \cite{maderna1979symmetrization} derived pointwise comparison results for weak solutions 
 under homogeneous Neumann conditions.
Ferone et al. \cite{ferone2005neumann} extended Talenti comparisons to cylindrical domains.
 Volzone \cite{volzone2016symmetrization} studied fractional Neumann problems and obtained similar results. 
For homogeneous Neumann problems on spherical shells, 
Langford \cite{langford2015symmetrization} established norm and oscillation comparisons under appropriate assumptions; 
see also \cite{langford2015neumann}.
More recently, Celentano et al.~\cite{celentano2025talenti}
considered nonhomogeneous Neumann problems and
obtained Lorentz norm comparison results under certain normalization assumptions.

Motivated by the above developments, we investigate
nonhomogeneous Neumann problems on Riemannian manifolds
with Ricci curvature bounded below.

Our framework includes both complete noncompact manifolds
with nonnegative Ricci curvature and positive asymptotic
volume ratio, and compact manifolds satisfying
\[
\mathrm{Ric}\ge (n-1).
\]
The corresponding comparison spaces are Euclidean balls
and spherical caps, respectively.
Motivated by the above developments, 
we investigate nonhomogeneous Neumann problems on Riemannian manifolds with nonnegative Ricci curvature.

Let $(M,g)$ be a complete noncompact Riemannian manifold with nonnegative Ricci curvature, 
and let $\Omega\subset M$ be a bounded domain with Lipschitz boundary. 
We consider the boundary value problem
\begin{equation}\label{eq:1.3}
\begin{cases}
-\operatorname{div}\bigl(|\nabla u|^{p-2}\nabla u\bigr)=f,
& \text{in }\Omega,
\\[1ex]
|\nabla u|^{p-2}\dfrac{\partial u}{\partial\nu}
+\beta^{p-1}=0,
& \text{on }\partial\Omega.
\end{cases}
\end{equation}

Following Talenti's philosophy, 
we compare solutions of \eqref{eq:1.3} with those of a suitable symmetrized problem posed on a model space. 
Let $\Omega^\sharp$ denote the symmetrized domain associated with $\Omega$ (see \ref{dom}), 
and let $f^\sharp$ be the symmetrization of $f$. We consider the problem
\begin{equation}\label{eq:1.4}
\begin{cases}
-\operatorname{div}\bigl(|\nabla v|^{p-2}\nabla v\bigr)
=
f^\sharp,
& \text{in }\Omega^\sharp,
\\[1ex]
|\nabla v|^{p-2}
\dfrac{\partial v}{\partial\nu}
+
(\beta^\sharp)^{p-1}
=
0,
& \text{on }\partial\Omega^\sharp.
\end{cases}
\end{equation}

\begin{remark}
Unless otherwise specified, 
the symbol $|\cdot|$ always denotes the corresponding Riemannian measure.
 Likewise, all differential operators, gradients, volume forms, 
 and boundary measures are understood with respect to the ambient manifold under consideration.
\end{remark}
To guarantee the solvability of the Neumann problems, the compatibility condition must be satisfied. Integrating the equations and applying the divergence theorem, we obtain
\begin{equation}
\beta^{p-1}
=
\frac{1}{|\partial\Omega|}
\int_\Omega f\,dV,
\qquad
(\beta^\sharp)^{p-1}
=
\frac{1}{|\partial\Omega^\sharp|}
\int_{\Omega^\sharp}
f^\sharp\,dV.
\end{equation}
We further assume that
\[
f \ge 0, \int_\Omega f\,dV > 0.
\]
Under this assumption, we have
\[
\beta^{p-1},\, (\beta^\sharp)^{p-1} > 0,
\]
and therefore both $\beta$ and $\beta^\sharp$ can be taken to be positive.
 
We know that if $u$ is a solution, then $u+C$ is also a solution for any constant $C$. 
Therefore, one may always normalize the solution by adding a suitable constant, and in particular, 
we may assume that the solution is positive.
However, this normalization alone is not sufficient to establish a meaningful comparison principle between solutions on $\Omega$ and $\Omega^\sharp$. 
Indeed, due to the non-uniqueness up to constants, an additional normalization condition is required to ensure a one-to-one correspondence between the two solutions.
In the present paper, we impose the following compatibility (normalization) condition:
\begin{align}
\frac{\beta^\sharp}{\beta}
\int_{\partial\Omega}
u^q\,d\sigma
&=
\theta_\kappa
\int_{\partial\Omega^\sharp}
v^q\,d\sigma.
\label{eq:1.6}
\end{align}
with $\theta_\kappa$ given by \eqref{theta}.

To formulate the normalization conditions, we introduce the boundary $L^q$-moments
\[
\alpha_q
=
\frac{1}{q}
\int_{\partial\Omega}
u^q\,d\sigma,
\qquad
\alpha_q^\sharp
=
\frac{1}{q}
\int_{\partial\Omega^\sharp}
v^q\,d\sigma.
\]

In terms of these quantities, the above relations can be rewritten as
\begin{align}\label{eq:1.7}
\frac{\beta^\sharp}{\beta}\,\alpha_q
&=
\theta_\kappa\,\alpha_q^\sharp.
\end{align}
We emphasize that the parameter $q$ is completely independent of $p$, 
and the normalization condition can be imposed for every $q\ge 1$. 

One of the observations of the present paper is that the comparison argument is not restricted to particular choices of $q$. 
This additional flexibility appears to be a characteristic feature of Neumann boundary value problems, 
where the invariance of solutions under additive constants allows a broader class of admissible normalizations.

Such a relaxation is no longer available in the Robin case, 
where the boundary condition imposes a much stronger constraint on the admissible normalizations; see \cite{amato2022comparison}.

We are now ready to state our main result.
\begin{theorem}\label{thm:1.2}
Let $q\ge 1$, and let $u$ be a positive weak solution to problem \eqref{eq:1.3}.

Suppose that $v$ is the positive solution to problem \eqref{eq:1.4}
satisfying condition \eqref{eq:1.7}. Then
\begin{align}\label{eq:1.8}
\|u\|_{L^{\lambda,1}(\Omega)}
\le
\theta_\kappa^{\frac{1}{\lambda}}
\|v\|_{L^{\lambda,1}(\Omega^\sharp)},
\qquad
\kappa=0,1,
\quad
0 < \lambda \le \frac{n(p-1)}{p(n-1)},
\end{align}
and
\begin{align}\label{eq:1.9}
\|u\|_{L^{q\lambda,q}(\Omega)}
\le
\theta_\kappa^{\frac{1}{q\lambda}}
\|v\|_{L^{q\lambda,q}(\Omega^\sharp)},
\qquad
\kappa=0,
\quad
0 < \lambda \le \frac{n(p-1)}{p(n-2)+n},
\end{align}
while
\begin{align}\label{eq:1.10}
\|u\|_{L^{q\lambda,q}(\Omega)}
\le
\theta_\kappa^{\frac{1}{q\lambda}}
\|v\|_{L^{q\lambda,q}(\Omega^\sharp)},
\qquad
\kappa=1,
\quad
0 < \lambda \le \frac{n(p-1)}{(p+1)(n-1)}.
\end{align}
In particular, if $n=2$, $\kappa=1$, and $0 < \lambda \le p-1$, then
\begin{align}\label{eq:1.11}
\|u\|_{L^{q\lambda,q}(\Omega)}
\le
\theta_\kappa^{\frac{1}{q\lambda}}
\|v\|_{L^{q\lambda,q}(\Omega^\sharp)}.
\end{align}
\end{theorem}
\begin{theorem}\label{thm:1.3}
Assume that $f\equiv1$ in $\Omega$.
Let $q\ge 1$, and let $u$ be a positive weak solution to problem \eqref{eq:1.3}.

Suppose that $v$ is the positive solution to problem \eqref{eq:1.4}
satisfying condition \eqref{eq:1.7}. If
\[
1\le p\le \frac{n}{n-1},
\]
then
\begin{align}
u^\sharp(x)\le \theta_\kappa v(x),
\qquad
\forall\,x\in\Omega^\sharp.
\end{align}

If
\[
p > \frac{n}{n-1},
\]
then the following estimates hold:
\begin{align}\label{eq:1.13}
\|u\|_{L^{\lambda,1}(\Omega)}
\le
\theta_\kappa^{\frac{1}{\lambda}}
\|v\|_{L^{\lambda,1}(\Omega^\sharp)},
\qquad
\kappa=0,1,
\quad
0<\lambda \le\frac{n(p-1)}{n(p-1)-p}.
\end{align}
Moreover, when $\kappa=0$, one has
\begin{align}\label{eq:1.14}
\|u\|_{L^{q\lambda,q}(\Omega)}
\le
\theta_\kappa^{\frac{1}{q\lambda}}
\|v\|_{L^{q\lambda,q}(\Omega^\sharp)},
\qquad
0<\lambda \le\frac{n(p-1)}{n(p-1)-p}.
\end{align}
\end{theorem}
\begin{remark}
In order to make the above estimate precise, 
we recall the definition of Lorentz spaces $L^{\lambda,q}(\Omega)$, 
where $0<\lambda<\infty$ and $0<q\le\infty$.
\begin{align*}
\|g\|_{L^{\lambda,q}(\Omega)}
=
\begin{cases}
  \lambda^{\frac{1}{q}}
\left(
\displaystyle \int_0^\infty
\big(t\,|\{x\in\Omega: |g(x)|>t\}|^{1/\lambda}\big)^q \frac{dt}{t}
\right)^{1/q}, & 0<q<\infty,\\[1em]
\sup_{t>0} t\,|\{x\in\Omega: |g(x)|>t\}|^\frac1\lambda, & q=\infty.
\end{cases}
\end{align*}
For further discussions on related topics, we refer the reader to
\cite{kaminska1990some,carro2007recent}.
\end{remark}
\section{Preliminaries}
\subsection{Manifolds and Isoperimetric Inequalities}
Let $M$ be a complete noncompact Riemannian manifold with nonnegative Ricci curvature. Assume that $M$ has positive asymptotic volume ratio, namely,
\begin{equation}
\theta
=
\lim_{r\to\infty}
\frac{|B(r)|}{|B^n|\,r^n}
>0,
\end{equation}
where $B(r)$ denotes the geodesic ball of radius $r$ in $M$, and $|B^n|$ is the volume of the unit ball in $\mathbb{R}^n$.

In \cite{brendle2023sobolev}, Brendle proved that for every bounded domain $\Omega\subset M$ with smooth boundary $\partial\Omega$, the following isoperimetric inequality holds:
Let $\Omega^\sharp$ denote the Euclidean ball in $\mathbb{R}^n$ satisfying
\[
\theta |\Omega^\sharp| = |\Omega|.
\]
Then
\[
|\partial \Omega^\sharp|
=
n|B^n|^{\frac1n}
|\Omega^\sharp|^{\frac{n-1}{n}},
\]
and consequently
\[
|\partial \Omega|
\ge
\theta |\partial \Omega^\sharp|.
\]
For further developments on isoperimetric inequalities, we refer the reader to
\cite{agostiniani2018sharp,cavalletti2017sharp,klartag2017needle}
and the references therein.
Now let $M$ be a compact Riemannian manifold satisfying
\[
\operatorname{Ric}_M \ge (n-1)k,
\qquad k>0.
\]
Denote by $S_k^n$ the standard sphere of radius $1/\sqrt{k}$, and let $\Omega^\sharp$ be the geodesic ball in $S_k^n$ such that
\[
\frac{|M|}{|S_k^n|}
|\Omega^\sharp|
=
|\Omega|.
\]
By the Levy--Gromov isoperimetric inequality\cite{gromov2007metric}, we have
\[
|\partial \Omega|
\ge
\frac{|M|}{|S_k^n|}
|\partial \Omega^\sharp|.
\]

By a suitable scaling of the metric, it suffices to consider the case $k=1$. To unify the notation, we define
\begin{align}{\label{theta}}
\theta_\kappa=
\begin{cases}
\displaystyle
\lim_{r\to\infty}
\frac{|B(r)|}{|B^n|r^n},
& \kappa=0,
\\[2ex]
\displaystyle
\frac{|M|}{|S^n|},
& \kappa=1.
\end{cases}
\end{align}
For convenience, we denote by $M_\kappa$ the model space, namely,
\[
M_\kappa=
\begin{cases}
\mathbb{R}^n, & \kappa=0,\\
S^n, & \kappa=1,
\end{cases}
\]
equipped with their standard metrics.

Throughout this paper, $M$ denotes a Riemannian manifold satisfying
\[
\operatorname{Ric}_M \ge (n-1)\kappa,
\qquad
\kappa\in\{0,1\}.
\]
In the case $\kappa=0$, we additionally assume that $M$ has positive asymptotic volume ratio.

Under these assumptions, for every measurable set
$\Omega\subset M$, we denote by $\Omega^\sharp$ the geodesic ball in the model space $M_\kappa$ such that
\begin{equation}\label{dom}
|\Omega|
=
\theta_\kappa |\Omega^\sharp|.
\end{equation}
With this normalization, the corresponding isoperimetric inequalities take the unified form
\begin{equation}\label{eq:isoperimetric}
|\partial\Omega|
\ge
\theta_\kappa\,|\partial\Omega^\sharp|.
\end{equation}
Moreover,
equality holds if and only if $M$ is isometric to the model space $M_\kappa$ and $\Omega$ is a geodesic ball in $M_\kappa$.
\subsection{Rearrangements}
Talenti's classical approach relies on spherical symmetrization, which is more commonly referred to as the rearrangement method in the modern literature. 
In this section, we briefly recall the basic notation and several fundamental properties of rearrangements that will be used throughout the paper.
We note that these ideas also naturally arise in convex geometry. For further references, see \cite{gruber2007convex},\cite{schneider2013convex}.
\begin{definition}
Let $w:\Omega\to[0,\infty)$ be a measurable function. For $t\ge0$, define
\[
\Omega_{w,t}
=
\{x\in\Omega:|w(x)|>t\},
\qquad
\mu_w(t)
=
|\Omega_{w,t}|.
\]
The decreasing rearrangement $w^*$ of $w$ is defined by
\begin{align}
w^*(s)
=
\begin{cases}
\operatorname*{ess\,sup}_{x\in\Omega}w(x),
& s=0,
\\[0.3em]
\inf\{t\ge0:\mu_w(t)<s\},
& s>0,
\end{cases}
\label{eq:2.2}
\end{align}
for $s\in[0,|\Omega|]$.
The Schwarz decreasing rearrangement of $w$ is the radially symmetric function $w^\sharp:\Omega^\sharp\to[0,\infty)$ defined by
\begin{align}
w^\sharp(x)
=
w^*\bigl(
\theta_\kappa|B_{\kappa}(|x|)|\
\bigr),
\qquad
x\in\Omega^\sharp,
\end{align}
where $B_{\kappa}(|x|)$ denotes the geodesi ball in ${M}_\kappa$ with radius $|x|$. 
\end{definition}
We emphasize that our definition of spherical rearrangement contains the normalization factor $\theta_\kappa$. 
Consequently, compared with the classical Euclidean rearrangement, 
the distribution function and several related identities acquire an additional factor $\theta_\kappa$. 
We refer the reader to \cite{chen2021talenti} for further discussion.
The distribution functions of $w$ and its rearrangement $w^\sharp$ satisfy
\begin{equation}
\mu_w(t)
=
\theta_\kappa \mu_{w^\sharp}(t).
\end{equation}

As a consequence, the decreasing rearrangement $w^*$ and the spherical rearrangement $w^\sharp$ are related through the identity
\[
\mu_w(t)
=
\mu_{w^*}(t)
=
\theta_\kappa \mu_{w^\sharp}(t).
\]

The functions $w$, $w^*$ and $w^\sharp$ are equimeasurable in the sense that
\begin{equation}\label{eq:2.8}
\|w\|_{L^p(\Omega)}
=
\|w^*\|_{L^p(0,|\Omega|)}
=
\theta_\kappa^{\frac{1}{p}}
\|w^\sharp\|_{L^p(\Omega^\sharp)}.
\end{equation}
Given measurable functions $u$, $v$ on $\Omega$, the Hardy-Littlewood inequality holds,
\begin{align}\label{eq:2.6}
  \int_{\Omega} |uv| \,dV \le \int_{0}^{|\Omega|} u^*(s)v^*(s) \,ds.
\end{align}
Choosing $v = \chi_{\{|u| > t \}} $ in \eqref{eq:2.6}, one has
\begin{align}
   \int_{\Omega} |u|\,dV \le \int_{0}^{\mu_(t)} u^*(s)\, ds
\end{align}
For $U\subset\Omega$, we define
\[
\partial U^i := \partial U \cap \Omega,
\qquad
\partial U^e := \partial U \cap \partial\Omega.
\]

Let $u$ and $v$ be solutions to problems \eqref{eq:1.3} and \eqref{eq:1.4}, respectively. For $t\in\mathbb{R}$, we define
\begin{align*}
U_t &= \{x\in\Omega: |u(x)|>t\}, \qquad \mu(t)=|U_t|,\\
V_t &= \{x\in\Omega^\sharp: |v(x)|>t\}, \qquad  \varphi(t)=|V_t|.
\end{align*}
\section{Proof of main result}
For the sake of notation, we introduce several geometric quantities associated with the model space $M_\kappa$, involving only basic Riemannian geometric concepts.

We define the function $S_\kappa$ by
\[
S_\kappa(l)
=
\frac{dL_\kappa}{dr}\circ L_\kappa^{-1}(l),
\]
where
\[
L_\kappa(r)
=
n|B^n|\theta_\kappa
\int_0^r \mathrm{sn}_\kappa^{\,n-1}(s)\,ds,
\]
and
\[
\mathrm{sn}_\kappa(r)=
\begin{cases}
r, & \kappa=0,\\
\sin r, & \kappa=1.
\end{cases}
\]

The function $L_\kappa$ represents the volume of a geodesic ball of radius $r$ in the model space $M_\kappa$, while $S_\kappa$ describes the corresponding surface area expressed as a function of the volume variable $l$. The factor $\theta_\kappa$ is introduced for normalization and simplifies subsequent computations.

To establish our main results, the following lemma plays a crucial role.
\begin{lemma}[Gronwall inequality]
\label{lem:3.1}

Let $\zeta:[\tau_0,\infty)\to\mathbb{R}$ be a continuously differentiable function satisfying
\[
\tau\,\zeta'(\tau)
\le
(q-1)\zeta(\tau)+C,
\qquad
\forall\,\tau\ge\tau_0>0,
\]
for some constant $C\ge0$.

Then, for every $\tau\ge\tau_0$,
\[
\zeta(\tau)
\le
\left(
\zeta(\tau_0)+\frac{C}{q-1}
\right)
\left(
\frac{\tau}{\tau_0}
\right)^{q-1}
-\frac{C}{q-1}.
\]

Consequently,
\[
\zeta'(\tau)
\le
\frac{(q-1)\zeta(\tau_0)+C}{\tau_0}
\left(
\frac{\tau}{\tau_0}
\right)^{q-2}.
\]
\end{lemma}
The following lemma is the main ingredient of this paper. Its idea originates from the work of Alvino et al.\cite{alvino2023talenti}, 
and it also serves as a fundamental tool in many comparison results for Robin boundary value problems.
\begin{lemma}
  Let $u$ and $v$ be positive solutions of the above problems. For $t>0$, we have the following result:
\begin{align}
S_\kappa(\mu(t))^{\frac{p}{p-1}}
&\le
\left(\int_0^{\mu(t)} f^*(s)\,ds\right)^\frac{1}{p-1}
\left(-\mu'(t)+\frac{1}{\beta}\int_{\partial U_t^e} d\sigma\right),
\label{eq:3.1}
\\
\tilde S_\kappa(\varphi(t))^\frac{p}{p-1}
&=
\theta_\kappa^{-\frac{1}{p-1}}
\left(\int_0^{\theta_\kappa \varphi(t)} f^*(s)\,ds\right)^{\frac{1}{p-1}}
\left(-\varphi'(t)+\frac{1}{\beta^\sharp}\int_{\partial V_t^e} d\sigma\right),
\label{eq:3.2}
\end{align}
\end{lemma}
where $\tilde{S_\kappa}(l) = \theta_\kappa^{-1}S_\kappa(\theta_\kappa l)$.
\begin{proof}
Let $t>0$ and $c>0$. We introduce the standard truncation function associated with $u$ at levels $t$ and $t+c$, 
defined by
\[
\psi(x)=
\begin{cases}
0, & u(x)\le t,\\[1ex]
u(x)-t, & t<u(x)<t+c,\\[1ex]
c, & u(x)\ge t+c.
\end{cases}
\]
This function is Lipschitz and belongs to $W^{1,p}(\Omega)$ whenever $u\in W^{1,p}(\Omega)$.
Choosing $\psi$ as a test function in the weak formulation, we obtain
\begin{align*}
\int_{U_t\setminus U_{t+c}} |\nabla u|^p\,dV
&+\beta^{p-1} c\int_{\partial U_{t+c}^e} d\sigma
+\beta^{p-1}\int_{\partial U_t^e\setminus \partial U_{t + c }^e}(u-t)\,d\sigma
\\
&=
\int_{U_t\setminus U_{t+c}} f(u-t)\,dV
+c\int_{U_{t+c}} f\,dV.
\end{align*}

Dividing by $c$ and letting $c\to0$, using the coarea formula 
\[
-\mu'(t)
=
\int_{\partial U_t^i}\frac{1}{|\nabla u|}\,d\sigma.
\]
We obtain for a.e. $t>0$,
\[
\int_{\partial U_t^i} |\nabla u|^{p-1}\,d\sigma
+\beta\int_{\partial U_t^e} d\sigma
=
\int_{U_t} f\,dx.
\]

Set
\[
h(x)=
\begin{cases}
|\nabla u|^{p-1}, & x\in \partial U_t^i,\\
\beta^{p-1}, & x\in \partial U_t^e.
\end{cases}
\]

Then Hölder's inequality gives
\begin{align*}
|\partial U_t|
&\le
\left(\int_{\partial U_t} h\,d\sigma\right)^\frac{1}{p}
\left(\int_{\partial U_t} \frac{1}{h^\frac{1}{p-1}}\,d\sigma\right)^\frac{p-1}{p}
\\
&\le \left(\int_{\partial U_t}h\,d\sigma\right)^\frac{1}{p}\left(\int_{\partial U_t^i}\frac{1}{|\nabla u|}\,d\sigma + \frac{1}{\beta}\int_{\partial U_t^e}\,d\sigma\right)^\frac{p-1}{p}
\\
&\le \left(\int_{0}^{\mu(t)}f^*(s)ds\right)^\frac{1}{p}\left(-\mu'(t) + \frac{1}{\beta}\int_{\partial U_t^e}\,d\sigma\right)^\frac{p-1}{p}
\end{align*}

Combining this with the isoperimetric inequality yields
\[
S_\kappa(\mu(t))^{\frac{p}{p-1}}
\le
|\partial U_t|^{\frac{p}{p-1}}
\le
\left(\int_0^{\mu(t)} f^*(s)\,ds\right)^\frac{1}{p-1}
\left(-\mu'(t)+\frac{1}{\beta}\int_{\partial U_t^e} d\sigma\right).
\]
The argument for $v$ is analogous. Indeed, since $v$ is radial, nonincreasing and $\Omega^\sharp$ is a geodesic ball, 
all the above inequalities become equalities. In particular, we have
\[
\tilde{S}_\kappa(\varphi(t)) = |\partial V_t|,
\qquad
\int_{V_t} f^\sharp\,dV
=
\theta_\kappa^{-1}
\int_0^{\theta_\kappa \varphi(t)} f^*(s)\,ds.
\]

This completes the proof.

\end{proof}
We next observe an important property of the radial solution $v$ and $u$. In particular, 
$v$ and $u$ both attains its minimum on the boundary, and moreover, $v$
 is constant on the boundary. Based on this observation, 
we first establish a comparison result for minima under condition \eqref{eq:1.7}.
\noindent
We define
\[
u_m = \inf_{\Omega} u,
\qquad
v_m = \inf_{\Omega^\sharp} v.
\]
\noindent
We compute as follows:
\begin{align*}
\theta_\kappa v_m^{q} |\partial\Omega^\sharp|
&=
\theta_\kappa \int_{\partial\Omega^\sharp} v^{q}\,d\sigma
\\
&=
\frac{\beta^\sharp}{\beta}
\int_{\partial\Omega} u^{q}\,d\sigma\quad \text{ Condition \eqref{eq:1.7}}
\\
&\ge
\frac{\beta^\sharp}{\beta}\, u_m^{q} |\partial\Omega|
\\
&\ge
\left(\frac{|\partial\Omega|}{\theta_\kappa |\partial\Omega^\sharp|}\right)^{\frac{1}{p-1}} 
u_m^{q} |\partial\Omega|\quad \text{ Definition of } \beta \text{,} \beta^\sharp \text{ and equation}~\eqref{eq:2.8}
\\
&\ge 0,
\end{align*}
\noindent
consequently, we obtain
\[
v_m^{q}
\ge
\left(\frac{|\partial\Omega|}{\theta_\kappa |\partial\Omega^\sharp|}\right)^{1+\frac{1}{p-1}}
u_m^{q}.
\]
\noindent
Finally, by the isoperimetric inequality~\eqref{eq:isoperimetric},
\[
|\partial\Omega|
\ge
\theta_\kappa |\partial\Omega^\sharp|,
\]
which yields
\[
v_m \ge u_m.
\]
Moreover, we observe that
\begin{align}\label{eq:3.3}
\mu(t) \le \theta_\kappa \varphi(t),
\qquad
\text{for all } 0\le t \le v_m.
\end{align}
We therefore obtain the following result:
\begin{lemma}
Let $u$ and $v$ be positive solutions of \eqref{eq:1.3} and \eqref{eq:1.4}, respectively. Then, for all $\tau\ge v_m$, we have
\begin{align*}
\int_0^{\tau}t^{q-1}\int_{\partial V_t^e} d\sigma dt
=
\frac{1}{q}\int_{\partial\Omega^\sharp} v^q\, d\sigma
= \alpha_q^\sharp
\end{align*}
while
\begin{align*}
\int_0^{\tau}t^{q-1}\int_{\partial U_t^e} d\sigma dt
\le
\frac{1}{q}\int_{\partial\Omega} u^q\, d\sigma
=\alpha_q
\end{align*}
\end{lemma}

\begin{proof}
By Fubini's theorem and the layer-cake representation, we obtain
\begin{align*}
\int_0^{\infty}t^{q-1}\int_{\partial V_t^e} d\sigma dt
&=
\int_{\partial\Omega^\sharp}\int_0^{v(x)} t^{q-1}dt\, d\sigma
=
\frac{1}{q}\int_{\partial\Omega^\sharp} v^{q}\, d\sigma,
\\
\int_0^{\infty}t^{q-1}\int_{\partial U_t^e} d\sigma dt
&=
\int_{\partial\Omega}\int_0^{u(x)}t^{q-1} dt\, d\sigma
=
\frac{1}{q}\int_{\partial\Omega} u^q\, d\sigma.
\end{align*}

Since the integrands are nonnegative, we clearly have
\[
\int_0^{\tau}t^{q-1}\int_{\partial U_t^e} d\sigma dt
\le
\int_0^{\infty}t^{q-1}\int_{\partial U_t^e} d\sigma dt.
\]

Moreover, for $\tau\ge v_m$, we have $\partial V_\tau^e \cap \partial \Omega^\sharp=\emptyset$ for $\tau>v_m$, hence
\[
\int_0^{\tau}t^{q-1}\int_{\partial V_t^e} d\sigma dt
=
\int_0^{\infty}t^{q-1}\int_{\partial V_t^e} d\sigma dt.
\]
This completes the proof.

\end{proof}
We are now ready to prove the theorem.
\begin{proof}[Proof of Theorem~\ref{thm:1.2}]
  We consider the positive solutions $u$ and $v$ to \eqref{eq:1.3} and \eqref{eq:1.4}, respectively, 
  which satisfy condition \eqref{eq:1.7}.

Let $0 < \lambda\le \frac{n(p-1)}{p(n-1)}$.

Multiplying inequality \eqref{eq:3.1} by
$t^{q-1}\,\mu(t)^{\frac{1}{\lambda}}S_\kappa(\mu(t))^{-\frac{p}{p-1}}$ and integrating over $[0,\tau]$ with $\tau>v_m$, we obtain
\begin{align*}
\int_0^\tau t^{q-1}\,\mu(t)^{\frac{1}{\lambda}}\,dt
&\le
\int_0^\tau
-t^{q-1}\,\mu'(t)\,\mu(t)^{\frac{1}{\lambda}}S_\kappa(\mu(t))^{-\frac{p}{p-1}}
\left(\int_0^{\mu(t)} f^*(s)\,ds\right)^{\frac{1}{p-1}}\,dt
\notag\\
&\quad
+\frac{1}{\beta}
\int_0^\tau
t^{q-1}\,\mu(t)^{\frac{1}{\lambda}}S_\kappa(\mu(t))^{-\frac{p}{p-1}}
\left(\int_{\partial U_t^e} d\sigma\right)
\left(\int_0^{\mu(t)} f^*(s)\,ds\right) ^\frac{1}{p-1}dt.
\end{align*}
\noindent
By the monotonicity of $l^{1/\lambda}S_\kappa(l)^{-\frac{p}{p-1}}$ for $0 < \lambda\le \frac{n(p-1)}{p(n-1)}$~(see Appendix~\ref{app:1}) and Lemma~\ref{lem:3.1}, we deduce
\begin{align*}
&\frac{1}{\beta}
\int_0^\tau t^{q-1}\,\mu(t)^{\frac{1}{\lambda}}S_\kappa(\mu(t))^{-\frac{p}{p-1}}
\left(\int_{\partial U_t^e} d\sigma\right)
\left(\int_0^{\mu(t)} f^*(s)\,ds\right) ^\frac{1}{p-1}
\\
&\le
\frac{\alpha_q}{\beta}
|\Omega|^{\frac{1}{\lambda}}S_\kappa(|\Omega|)^{-\frac{p}{p-1}}
\left(\int_0^{|\Omega|} f^*(s)\,ds\right)^{\frac{1}{p-1}}dt.
\end{align*}
Therefore, we have
\begin{align}\label{eq:3.5}
\int_0^\tau t^{q-1}\,\mu(t)^{\frac{1}{\lambda}}\,dt
&\le
\int_0^\tau
-t^{q-1}\,\mu'(t)\,\mu(t)^{\frac{1}{\lambda}}S_\kappa(\mu(t))^{-\frac{p}{p-1}}
\int_0^{\mu(t)} f^*(s)\,ds\,dt
\notag\\
&\quad
+\frac{\alpha_q}{\beta}
|\Omega|^{\frac{1}{\lambda}}S_\kappa(|\Omega|)^{-\frac{p}{p-1}}
\left(\int_0^{|\Omega|} f^*(s)\,ds\right)^{\frac{1}{p-1}}.
\end{align}
Define
\[
F_\kappa(l)
:=
\int_0^l
w^{\frac{1}{\lambda}}S_\kappa(w)^{-\frac{p}{p-1}}
\left(\int_0^w f^*(s)\,ds\right)^\frac{1}{p-1}\,dw.
\]
An integration by parts yields
\begin{align*}
\tau^{q-1}\left(
  \int_0^\tau \mu(t)^\frac{1}{\lambda}dt + F_\kappa(\mu(\tau))
\right)
&\le (q-1)\int_0^\tau t^{q-2}\left(\int_0^t \mu(r)^\frac1\lambda dr +F_\kappa(\mu(t))\right)dt\\
& +\frac{\alpha_q}{\beta}
|\Omega|^{\frac{1}{\lambda}}S_\kappa(|\Omega|)^{-\frac{p}{p-1}}
\left(\int_0^{|\Omega|} f^*(s)\,ds\right)^{\frac{1}{p-1}}.
\end{align*}
We now apply Gronwall's inequality, taking
\begin{align*}
\zeta(\tau) = \int_0^\tau t^{q-2}\left(\int_0^t \mu(r)^\frac1\lambda dr +F_\kappa(\mu(t))\right)dt, \\
C= \frac{\alpha_q}{\beta}
|\Omega|^{\frac{1}{\lambda}}S_\kappa(|\Omega|)^{-\frac{p}{p-1}}
\left(\int_0^{|\Omega|} f^*(s)\,ds\right)^{\frac{1}{p-1}}. 
\end{align*}
By choosing $\tau=v_m$, we obtain
\begin{align*}
  \tau^{q-2}\left(
\int_0^\tau \mu(t)^{\frac{1}{\lambda}}\,dt + F_\kappa(\mu(\tau))\right)
\le
\frac{(q-1)\zeta(v_m)+C}{v_m}\left(\frac{\tau}{v_m}\right)^{q-2}
\end{align*}
Observe that
\begin{align*}
\zeta(v_m)
=
\int_0^{v_m} t^{q-2}\left(\int_0^t \mu(r)^\frac1\lambda dr +F_\kappa(\mu(t))\right)dt
\end{align*}
Similarly, we multiply both sides of \eqref{eq:3.2} by
$t^{q-1}\varphi(t)\tilde{S}_{\kappa}(\varphi(t))^{-\frac{p}{p-1}}$
and integrate the resulting identity over $(0,\tau)$, where
$\tau\ge v_m$. This yields
\begin{align*}
\int_0^\tau t^{q-1}\,&\varphi(t)^{\frac{1}{\lambda}}\,dt\\
&=
\theta_\kappa^{-\frac{1}{q-1}}
\int_0^\tau
-t^{q-1}\,\varphi'(t)\,\varphi(t)^{\frac{1}{\lambda}}\tilde{S_\kappa}(\varphi(t))^{-\frac{p}{p-1}}
\left(\int_0^{\theta_\kappa\varphi(t)} f^*(s)\,ds\right)^\frac{1}{p-1}\,dt
\notag\\
&\quad
+\frac{\theta_\kappa^{-\frac1{q-1}}}{\beta^\sharp}
\int_0^\tau
t^{q-1}\,\varphi(t)^{\frac{1}{\lambda}}\tilde{S_\kappa}(\varphi(t))^{-\frac{p}{p-1}}
\left(\int_{\partial V_t^e} d\sigma\right)
\left(\int_0^{\varphi(t)} f^*(s)\,ds\right) ^\frac{1}{p-1}dt.
\end{align*}
By the definition of $\tilde{S}_{\kappa}(l)$, 
the above expression can be rewritten as
\begin{align*}
\int_0^\tau t^{q-1}\,&\varphi(t)^{\frac{1}{\lambda}}\,dt\\
&=
\theta_\kappa
\int_0^\tau
-t^{q-1}\,\varphi'(t)\,\varphi(t)^{\frac{1}{\lambda}}{S_\kappa}(\theta_\kappa\varphi(t))^{-\frac{p}{p-1}}
\left(\int_0^{\theta_\kappa\varphi(t)} f^*(s)\,ds\right)^{\frac{1}{p-1}}\,dt
\notag\\
&\quad
+\frac{\theta_\kappa}{\beta^\sharp}
\int_0^\tau
t^{q-1}\,\varphi(t)^{\frac{1}{\lambda}}{S_\kappa}(\theta_\kappa\varphi(t))^{-\frac{p}{p-1}}
\left(\int_{\partial V_t^e} d\sigma\right)
\left(\int_0^{\theta_\kappa\varphi(t)} f^*(s)\,ds\right) ^\frac{1}{p-1}dt.
\end{align*}
Hence by Lemma~\ref{lem:3.1}, we have
\begin{align}\label{eq:3.6}
\int_0^\tau t^{q-1}\,\varphi(t)^{\frac{1}{\lambda}}\,dt
&=
\theta_\kappa^{-\frac{1}{\lambda}}
\int_0^\tau
-t^{q-1}\,(F_\kappa(\theta_\kappa\varphi(t)))'\,dt
\notag\\
&\quad
+\theta_\kappa^{1-\frac{1}{\lambda}}\frac{\alpha_q^\sharp}{\beta^\sharp}
|\Omega|^{\frac{1}{\lambda}}S_\kappa(|\Omega|)^{-\frac{p}{p-1}}
\left(\int_0^{|\Omega|} f^*(s)\,ds\right)^{\frac{1}{p-1}}.
\end{align}
An integration by parts yields
\begin{align*}
\tau^{q-1}\Biggl(
  \int_0^\tau &(\theta_\kappa\varphi(t))^\frac{1}{\lambda}dt + F_\kappa(\theta_\kappa\varphi(\tau))
\Biggr)\\
&= (q-1)\int_0^\tau t^{q-2}\left(\int_0^t (\theta_\kappa\varphi(r))^\frac1\lambda dr +F_\kappa(\theta_\kappa\varphi(t))\right)dt\\
 &\qquad\qquad\qquad+\theta_\kappa\frac{\alpha_q^\sharp}{\beta^\sharp}
|\Omega|^{\frac{1}{\lambda}}S_\kappa(|\Omega|)^{-\frac{p}{p-1}}
\left(\int_0^{|\Omega|} f^*(s)\,ds\right)^{\frac{1}{p-1}}. 
\end{align*}
We now apply Gronwall's inequality again with
\begin{align*}
\zeta(\tau) = \int_0^\tau t^{q-2}\left(\int_0^t (\theta_\kappa\varphi(r))^\frac1\lambda dr +F_\kappa(\theta_\kappa\varphi(t))\right)dt, \\
C^\sharp= \theta_\kappa\frac{\alpha_q^\sharp}{\beta^\sharp}
|\Omega|^{\frac{1}{\lambda}}S_\kappa(|\Omega|)^{-\frac{p}{p-1}}
\left(\int_0^{|\Omega|} f^*(s)\,ds\right)^{\frac{1}{p-1}}. 
\end{align*}
In view of \eqref{eq:1.7}, we obtain
\[
C=C^\sharp.
\]
Furthermore,
\begin{align*}
 \tau^{q-2}\left(\int_0^\tau (\theta_\kappa\varphi(r))^\frac1\lambda dr 
+F_\kappa(\theta_\kappa\varphi(t))\right)dt
=
\frac{(q-1)\zeta(v_m)+C}{v_m}\left(\frac{\tau}{v_m}\right)^{q-2}
\end{align*}
At this point,
\begin{align*}
  \zeta(v_m)
  =
\int_0^{v_m} t^{q-2}\left(\int_0^t (\theta_\kappa\varphi(r))^\frac1\lambda dr 
+F_\kappa(\theta_\kappa\varphi(t))\right)dt
\end{align*}
We obtain the following inequality by inequality \eqref{eq:3.3}:
\begin{align*}
  \int_0^{v_m} t^{q-2}\left(\int_0^t \mu(r)^\frac1\lambda dr +F_\kappa(\mu(t))\right)dt
  \le
    \int_0^{v_m} t^{q-2}\left(\int_0^t (\theta_\kappa\varphi(r))^\frac1\lambda dr 
    +F_\kappa(\theta_\kappa\varphi(t))\right)dt
\end{align*}
Summarizing the above arguments, we obtain

\[
F_\kappa(\mu(\tau)) + \int_0^\tau \mu(t)^{\frac{1}{\lambda}}\,dt
\le
F_\kappa(\theta_\kappa\varphi(\tau)) + \int_0^\tau (\theta_\kappa\varphi(t))^{\frac{1}{\lambda}}\,dt.
\]

Letting $\tau\to\infty$ yields
\[
\int_0^\infty \mu(t)^{\frac{1}{\lambda}}\,dt
\le
\theta_\kappa^{\frac{1}{\lambda}}
\int_0^\infty \varphi(t)^{\frac{1}{\lambda}}\,dt,
\]
which implies
\[
\|u\|_{L^{\lambda,1}(\Omega)}
\le
\theta_\kappa^{\frac{1}{\lambda}}
\|v\|_{L^{\lambda,1}(\Omega^\sharp)}.
\]
We now prove inequalities \eqref{eq:1.9}, \eqref{eq:1.10}and \eqref{eq:1.11}. These can be rewritten as
\[
\int_0^\infty t^{q-1}\,\mu(t)^{\frac{1}{\lambda}}\,dt
\le
\int_0^\infty t^{q-1}\,(\theta_\kappa \varphi(t))^{\frac{1}{\lambda}}\,dt.
\]
\noindent
Integrate \eqref{eq:3.5} by parts on the right with the first term and letting $\tau\to\infty$, we obtain
\begin{align*}
\int_0^\infty t^{q-1}\,\mu(t)^{\frac{1}{\lambda}}\,dt
\le
(q-1)
\int_0^\infty t^{q-2} &F_\kappa(\mu(t))\,dt
\\+\frac{\alpha_q}{\beta}&
|\Omega|^{\frac{1}{\lambda}}S_\kappa(|\Omega|)^{-\frac{p}{p-1}}
\left(\int_0^{|\Omega|} f^*(s)\,ds\right)^{\frac{1}{p-1}}.
\end{align*}
\noindent
On the other hand,
\begin{align*}
\int_0^\infty t^{q-1}\,(\theta_\kappa \varphi(t))^{\frac{1}{\lambda}}\,dt
=
(q-1)
\int_0^\infty t^{q-2}&F_\kappa(\theta_\kappa \varphi(t))\,dt
\\+
\theta_\kappa &\frac{\alpha_q^\sharp}{\beta^\sharp}
|\Omega|^{\frac{1}{\lambda}}S_\kappa(|\Omega|)^{-\frac{p}{p-1}}
\left(\int_0^{|\Omega|} f^*(s)\,ds\right)^{\frac{1}{p-1}}.
\end{align*}

Therefore, it suffices to prove
\[
\int_0^\infty t^{q-2}F_\kappa(\mu(t))\,dt
\le
\int_0^\infty t^{q-2}F_\kappa(\theta_\kappa \varphi(t))\,dt.
\]

To this end, we multiply inequality \eqref{eq:3.1} by
$t^{q-1} F_\kappa(\mu(t)) S_\kappa(\mu(t))^{-\frac{p}{p-1}}$ and integrate over $[0,\tau]$ with $\tau \ge v_m$, obtaining
\begin{align*}
\int_0^\tau t^{q-1} &F_\kappa(\mu(t))\,dt\\
&\le
\int_0^\tau
-t^{q-1}\,\mu'(t)\,F_\kappa(\mu(t))S_\kappa(\mu(t))^{-\frac{p}{p-1}}
\left(\int_0^{\mu(t)} f^*(s)\,ds\right)^\frac{1}{p-1}\,dt
\\
&\quad
+\frac{1}{\beta}
\int_0^\tau
t^{q-1} F_\kappa(\mu(t))S_\kappa(\mu(t))^{-\frac{p}{p-1}}
\left(\int_{\partial U_t^e} d\sigma\right)
\left(\int_0^{\mu(t)} f^*(s)\,ds\right) ^\frac{1}{p-1}dt.
\end{align*}

We can prove that 
$l \mapsto F_0(l)S_0(l)^{-\frac{p}{p-1}}$ is non-decreasing for
$\lambda\in\left(0,\frac{n(p-1)}{p(n-2)+n}\right]$, and $ l \mapsto F_1(l)S_1(l)^{-\frac{p}{p-1}}$ is non-decreasing 
for $ 0 < \lambda \le \frac{n(p-1)}{(p+1)(n-1)}$, and we also deffer this prove in Appendix~\ref{app:1} for concise.\par

Since $F_\kappa(l)S_\kappa(l)^{-\frac{p}{p-1}}$ is non-decreasing, we have
\begin{align*}
\int_0^\tau t^{q-1} F_\kappa(\mu(t))\,dt
\le
-\int_0^\tau t^{q-1}\,&dH_\kappa(\mu(t))
\\+\frac{\alpha_q}{\beta}&
F_\kappa(|\Omega|)S_\kappa(|\Omega|)^{-\frac{p}{p-1}}
\left(\int_0^{|\Omega|} f^*(s)\,ds\right)^{\frac{1}{p-1}},
\end{align*}
where
\[
H_\kappa(l)
:=
\int_0^l
F_\kappa(w)S_\kappa(w)^{-\frac{p}{p-1}}
\left(\int_0^w f^*(s)\,ds\right)^\frac{1}{p-1}\,dw.
\]
Integrating by parts and applying Gronwall's Inequality with
\begin{align*}
\zeta(\tau) = \int_0^\tau \int_0^t r^{q-2} F_\kappa(\mu(r))\,dr + \int_0^\tau H_\kappa(\mu(r)) \,dr\\
C = \frac{\alpha_q}{\beta}
F_\kappa(|\Omega|)S_\kappa(|\Omega|)^{-\frac{p}{p-1}}
\left(\int_0^{|\Omega|} f^*(s)\,ds\right)^{\frac{1}{p-1}},
\end{align*}

we have
   \begin{align*}
\tau\left(\int_0^\tau  t^{q-2}F_\kappa(\mu(t))\,dt + H_\kappa(\mu(\tau))\right)
\le
\frac{(q-1)\zeta(v_m)+C}{v_m}\left(\frac{\tau}{v_m}\right)^{q-2}
.
\end{align*}
Note that, by \eqref{eq:3.3} we have:
\begin{align*}
\zeta(v_m)&=\int_0^{v_m}\int_0^t r^{q-2}F_\kappa(\mu(r))dr + \int_0^{v_m} H_\kappa(\mu(r)) dr
\\
&\le\int_0^{v_m}\int_0^t r^{q-2}F_\kappa(\theta_\kappa\varphi(r))dr + \int_0^{v_m} H_\kappa(\theta_\kappa\varphi(r)) dr
\end{align*}
Similarly, using \eqref{eq:3.2}, we obtain
\begin{align*}
\int_0^\tau t^{q-1} F_\kappa(\theta_\kappa \varphi(t))\,dt
=
-\int_0^\tau t^{q-1}\,&dH_\kappa(\theta_\kappa \varphi(t))
\\+\theta_\kappa\frac{\alpha_q^\sharp}{\beta^\sharp}&
F_\kappa(|\Omega|)S_\kappa(|\Omega|)^{-\frac{p}{p-1}}
\left(\int_0^{|\Omega|} f^*(s)\,ds\right)^{\frac{1}{p-1}}.
\end{align*}
Applying Gronwall's inequality and combining the above estimates, we obtain
\[
\int_0^\tau t^{q-2} F_\kappa(\mu(t))\,dt + H_\kappa(\mu(\tau))
\le
\int_0^\tau t^{q-2} F_\kappa(\theta_\kappa \varphi(t))\,dt + H_\kappa(\theta_\kappa \varphi(\tau)).
\]

Letting $\tau\to\infty$, we conclude the desired inequality.
\end{proof}
\begin{proof}[Proof of Theorem \ref{thm:1.3}]
  For the special case $f\equiv 1$, we first observe that
\[
\int_0^{\mu(t)} f\, ds = \mu(t), 
\qquad 
\int_0^{\theta_\kappa \varphi(t)} f\, ds = \theta_\kappa \varphi(t).
\]
Hence, equations \eqref{eq:3.2} and \eqref{eq:3.3} reduce to
\begin{align}
S_\kappa(\mu(t))^{\frac{p}{p-1}}
&\le
\mu(t)^{\frac{1}{p-1}}
\left(-\mu'(t)+\frac{1}{\beta}\int_{\partial U_t^e} d\sigma\right),
\label{eq:3.7}
\\
\tilde S_\kappa(\varphi(t))^{\frac{p}{p-1}}
&=
\varphi(t)^{\frac{1}{p-1}}
\left(-\varphi'(t)+\frac{1}{\beta^\sharp}\int_{\partial V_t^e} d\sigma\right),
\label{eq:3.8}
\end{align}
For \eqref{eq:3.7}, we multiply both sides by
$t^{q-1} S_\kappa(\mu(t))^{-\frac{p}{p-1}}$
and integrate over $(0,\tau)$, with $\tau \ge v_m$. This yields
\begin{align*}
\int_0^\tau t^{q-1}\,dt
&\le
\int_0^\tau (-\mu'(t))\, t^{q-1}\mu(t)^{\frac{1}{p-1}}
S_\kappa(\mu(t))^{-\frac{p}{p-1}}\,dt
\nonumber\\
&\quad +
\frac{1}{\beta}
\int_0^\tau t^{q-1}\mu(t)^{\frac{1}{p-1}}
S_\kappa(\mu(t))^{-\frac{p}{p-1}}
\int_{\partial U_t^e} d\sigma\, dt.
\end{align*}

Since the function
$l \mapsto l^{\frac{1}{p-1}} S_\kappa(l)^{-\frac{p}{p-1}}$
is nondecreasing for $1 \le p \le \frac{n}{n-1}$, we deduce that
\begin{align*}
\int_0^\tau t^{q-1}\,dt
\le
\int_0^\tau (-\mu'(t))\, t^{q-1}\mu(t)^{\frac{1}{p-1}}
S_\kappa(\mu(t))^{-\frac{p}{p-1}}dt
+
\frac{\alpha_q}{\beta}
|\Omega|^{\frac{1}{p-1}}
S_\kappa(|\Omega|)^{-\frac{p}{p-1}}.
\end{align*}
We apply the same argument to \eqref{eq:3.8}. Consequently, we obtain
\begin{align*}
\int_0^\tau t^{q-1}\,dt
=
\int_0^\tau (-\varphi'(t))\, t^{q-1}\varphi(t)^{\frac{1}{p-1}}
\tilde{S}_\kappa(\varphi(t))^{-\frac{p}{p-1}}\,dt
+\theta_\kappa\frac{\alpha_q^\sharp}{\beta^\sharp}
|\Omega|^{\frac{1}{p-1}}
S_\kappa(|\Omega|)^{-\frac{p}{p-1}}.
\end{align*}

Combining the above inequality with the corresponding estimate for $\mu(t)$, we deduce that
\begin{align*}
\int_0^\tau (-\mu'(t))\,t^{q-1}\mu(t)^{\frac{1}{p-1}}
S_\kappa(\mu(t))^{-\frac{p}{p-1}}\,dt
\ge
\int_0^\tau (-\varphi'(t))\,t^{q-1}\varphi(t)^{\frac{1}{p-1}}
\tilde{S}_\kappa(\varphi(t))^{-\frac{p}{p-1}}\,dt.
\end{align*}

An integration by parts yields
\begin{align*}
&-\tau^{p-1}\int_0^{\mu(\tau)} l^{\frac{1}{p-1}} S_\kappa(l)^{-\frac{1}{p-1}}\,dl
+ (p-1)\int_0^\tau t^{p-2}\int_0^{\mu(t)} l^{\frac{1}{p-1}} S_\kappa(l)^{-\frac{p}{p-1}}\,dl\,dt \\
&\ge
-\tau^{p-1}\int_0^{\theta_\kappa \varphi(\tau)} l^{\frac{1}{p-1}} S_\kappa(l)^{-\frac{1}{p-1}}\,dl
+ (p-1)\int_0^\tau t^{p-2}\int_0^{\theta_\kappa \varphi(t)} l^{\frac{1}{p-1}} {S}_\kappa(l)^{-\frac{p}{p-1}}\,dl\,dt.
\end{align*}
We apply Gronwall's inequality once again, with
\[
\zeta(\tau)
=
\int_0^\tau t^{p-2}
\int_{\theta_\kappa \varphi(t)}^{\mu(t)}
l^{\frac{1}{p-1}} S_\kappa(l)^{-\frac{p}{p-1}}\,dl\,dt,
\qquad C=0.
\]
It follows that
\begin{align*}
\tau^{p-2}
\int_{\theta_\kappa \varphi(\tau)}^{\mu(\tau)}
l^{\frac{1}{p-1}} S_\kappa(l)^{-\frac{p}{p-1}}\,dl
\le
\frac{(q-1)\zeta(v_m)}{v_m}
\left(\frac{\tau}{v_m}\right)^{q-2}.
\end{align*}

Since $\zeta(v_m)\le 0$ and the above inequality holds for every $\tau \ge v_m$, we deduce that
\[
\int_{\theta_\kappa \varphi(\tau)}^{\mu(\tau)}
l^{\frac{1}{p-1}} S_\kappa(l)^{-\frac{p}{p-1}}\,dl
\le 0,
\qquad \forall\, \tau \ge v_m.
\]

On the other hand, since
$l^{\frac{1}{p-1}} S_\kappa(l)^{-\frac{p}{p-1}} > 0$
in the integration range, it follows that
\[
\mu(\tau) \le \theta_\kappa \varphi(\tau),
\qquad \forall\, \tau \ge v_m.
\]

Combining this with the previous estimate, we conclude that
\[
\mu(t) \le \theta_\kappa \varphi(t), \qquad \forall\, t>0,
\]
and therefore \eqref{eq:1.11} holds.
\par
We now turn to the proof of \eqref{eq:1.13} and \eqref{eq:1.14}. 
\par
Multiplying \eqref{eq:3.7} by
$t^{q-1} S_\kappa(\mu(t))^{-\frac{p}{p-1}}$
and integrating over $(0,\tau)$, with $\tau \ge v_m$, we obtain
\begin{align*}
\int_0^\tau t^{q-1}\,\mu(t)^{\frac{1}{\lambda}}\,dt
&\le
\int_0^\tau -t^{q-1}\,\mu'(t)\,\mu(t)^{\frac{1}{p-1}+\frac{1}{\lambda}}S_\kappa(\mu(t))^{-\frac{p}{p-1}}\,dt
\notag\\
&\quad
+\frac{1}{\beta}
\int_0^\tau t^{q-1}\,\mu(t)^{\frac{1}{p-1}+\frac{1}{\lambda}}S_\kappa(\mu(t))^{-\frac{p}{p-1}}
\left(\int_{\partial U_t^e} d\sigma\right) dt.
\end{align*}

Using the monotonicity of $l^{\frac{1}{p-1}+\frac{1}{\lambda}}S_\kappa(l)^{-\frac{p}{p-1}}$ for $\kappa=0,1$ when $0 < \lambda\le \frac{n(p-1)}{n(p-1)-p}$. We can obtain that
\begin{align}\label{eq:3.13}
\int_0^\tau t^{q-1}\,(\mu(t))^{\frac{1}{\lambda}}\,dt
\le
-\int_0^\tau t^{q-1}\,dF_\kappa((\mu(t)))
+\frac{\alpha_q}{\beta}
|\Omega|^{1+\frac{1}{\lambda}}S_\kappa(|\Omega|)^{-\frac{p}{p-1}}.
\end{align}
Proceeding as in the case of \eqref{eq:3.7}, we consider \eqref{eq:3.8}. Multiplying both sides by
\[
t^{q-1}\tilde{S}_\kappa(\varphi(t))^{-\frac{p}{p-1}}
\]
and integrating over $(0,\tau)$, $\tau \ge v_m$, we obtain
\begin{align}\label{eq:3.14}
\int_0^\tau t^{q-1}\,(\theta_\kappa\varphi(t))^{\frac{1}{\lambda}}\,dt
=
-\int_0^\tau t^{q-1}\,dF_\kappa((\theta_\kappa\varphi(t)))
+\frac{\alpha_q}{\beta}
|\Omega|^{1+\frac{1}{\lambda}}S_\kappa(|\Omega|)^{-\frac{p}{p-1}}.
\end{align}
Following the similar arguments in proof of \eqref{eq:1.8}, we are able to prove \eqref{eq:1.13}.
\par
We now turn to the proof of \eqref{eq:1.14}. It suffices to show that
\begin{align*}
\int_0^\infty t^{q-1}\,\mu(t)^{\frac{1}{\lambda}}\,dt
\le
\int_0^\infty t^{q-1}\,(\theta_0\varphi(t))^{\frac{1}{\lambda}}\,dt.
\end{align*}

To this end, Integrate \eqref{eq:3.13} and \eqref{eq:3.14} by parts on the right with the first term and letting $\tau\to\infty$. This reduces the problem to proving the equivalent inequality
\[
\int_0^\infty t^{q-1}F_0(\mu(t))\,dt
\le
\int_0^\infty t^{q-1}F_0(\theta_0 \varphi(t))\,dt.
\]
We multiply both sides of \eqref{eq:3.7} by
$t^{q-1} F_0(\mu(t)) S_0(\mu(t))^{\frac{1}{p-1}}$
and integrate over $(0,\tau)$, where $\tau \ge v_m$. This yields
\begin{align*}
\int_0^\tau t^{q-1} F_0(\mu(t))\,dt
&\le
\int_0^\tau -t^{q-1}\mu'(t)\, F_0(\mu(t)) S_0(\mu(t))^{-\frac{1}{p-1}}
\mu(t)^{\frac{1}{p-1}}\,dt \\
&\quad +
\frac{1}{\beta}
\int_0^\tau t^{q-1}\mu(t)^{\frac{1}{p-1}} F_0(\mu(t))
S_0(\mu(t))^{-\frac{p}{p-1}}
\int_{\partial U_t^e} d\sigma\, dt.
\end{align*}

Since the function
\[
l \mapsto l^{\frac{1}{p-1}} F_0(l) S_0(l)^{-\frac{p}{p-1}}
\]
is nondecreasing for
$0 < \lambda \le \frac{n(p-1)}{n(p-1)-p}$,
we deduce that
\begin{align*}
\int_0^\tau t^{q-1} F_0(\mu(t))\,dt
&\le
\int_0^\tau -t^{q-1}\mu'(t)\, F_0(\mu(t)) S_0(\mu(t))^{-\frac{1}{p-1}}
\mu(t)^{\frac{1}{p-1}}\,dt \\
&\quad +
\frac{\alpha_q}{\beta}
|\Omega|^{-\frac{1}{p-1}}
F_0(|\Omega|)
S_0(|\Omega|)^{-\frac{p}{p-1}}.
\end{align*}
We next apply a similar argument to \eqref{eq:3.8}. Multiplying both sides of \eqref{eq:3.8} By
\[
t^{q-1} F_0(\theta_0 \varphi(t)) \tilde{S}_0(\varphi(t))^{-\frac{p}{p-1}}
\]
and integrating over $(0,\tau)$, with $\tau \ge v_m$, we obtain
\begin{align*}
\int_0^\tau t^{q-1} F_0(\theta_0\varphi(t))\,dt
&=
\int_0^\tau -t^{q-1}\varphi'(t)\,F_0(\theta_0\varphi(t))
\tilde{S}_0(\varphi(t))^{-\frac{1}{p-1}}
\varphi(t)^{\frac{1}{p-1}}\,dt \\
&\quad +
\frac{1}{\beta^\sharp}
\int_0^\tau t^{q-1}\varphi(t)^{\frac{1}{p-1}}
F_0(\theta_0\varphi(t))
\tilde{S}_0(\varphi(t))^{-\frac{p}{p-1}}
\int_{\partial V_t^e} d\sigma\,dt.
\end{align*}

This can be rewritten as
\begin{align*}
\int_0^\tau t^{q-1} F_0(\theta_0\varphi(t))\,dt
&=
\int_0^\tau -t^{q-1}\varphi'(t)\,F_0(\theta_0\varphi(t))
\tilde{S}_0(\varphi(t))^{-\frac{1}{p-1}}
\varphi(t)^{\frac{1}{p-1}}\,dt \\
&\quad +
\theta_0
\frac{\alpha_q^\sharp}{\beta^\sharp}
|\Omega|^{-\frac{1}{p-1}}
F_0(|\Omega|)
S_0(|\Omega|)^{-\frac{p}{p-1}}.
\end{align*}

We now return to the proof of \eqref{eq:1.9}. Arguing in the same way as above, we conclude that \eqref{eq:1.14} holds.
\end{proof}
\begin{remark}
 One may ask whether the constraints imposed in the comparison procedure are optimal. 
The Neumann case considered here indicates that weaker assumptions may still suffice to establish Talenti-type inequalities. 
Identifying the minimal set of assumptions under which such comparison results remain valid appears to be an interesting open problem.
\end{remark}

\appendix
\section*{Appendix}

\section{Proof of the monotonicity properties}
\label{app:1}
\begin{proof}[Proof of the monotonicity of $l^{1/\lambda}S_\kappa(l)^{-\frac{p}{p-1}}$]
For $\kappa=0$, define
\[
C_0
:=
n^{-\frac{p}{p-1}}
|B^n|^{-\frac{p}{n(p-1)}}
\theta_0^{-\frac{p}{n(p-1)}},
\] 
and we have
\[
l^{1/\lambda}S_0(l)^{-\frac{p}{p-1}}
= C_0\,
l^{\frac{1}{\lambda}-\frac{p(n-1)}{n(p-1)}},
\]
which is non-decreasing provided that $0 < \lambda\le \frac{n(p-1)}{p(n-1)}$.

For $\kappa=1$, we compute
\begin{align}\label{eq:derivative-FS}
\left(l^{1/\lambda}S_1(l)^{-\frac{p}{p-1}}\right)'
&=
\frac1\lambda
l^{\frac1\lambda-1}
S_1(l)^{-\frac{p}{p-1}}
-\frac{p}{p-1}
l^{\frac1\lambda}
S_1'(l)
S_1(l)^{-\frac{p}{p-1}-1}
\\\notag
&=
l^{\frac1\lambda-1}
S_1(l)^{-\frac{p}{p-1}-1}
\left[
\frac1\lambda S_1(l)
-\frac{p}{p-1}lS_1'(l)
\right]
\\\notag
&=
l^{\frac1\lambda-1}
S_1(l)^{-\frac{p}{p-1}-1}
\Biggl[
\frac1\lambda
n|B^n|\theta_1
\sin^{n-1}r
\\\notag
&\qquad
-\frac{p}{p-1}
\Bigl(
n|B^n|\theta_1
\int_0^r\sin^{n-1}s\,ds
\Bigr)
\frac{(n-1)\cos r}{\sin r}
\Biggr]
\\\notag
&=
l^{\frac1\lambda-1}
S_1(l)^{-\frac{p}{p-1}-1}
\bigl(n|B^n|\theta_1\sin^{n-1}r\bigr)
\\\notag
&\qquad\times
\left[
\frac1\lambda
-\frac{p(n-1)}{p-1}
\frac{\cos r}{\sin^n r}
\int_0^r\sin^{n-1}s\,ds
\right].
\end{align}
Since for $0 \le r \le \pi$
\[
\frac{\cos r}{\sin^n r}
\int_0^r \sin^{n-1}s\,ds
\le \frac1n,
\]
we deduce that
\[
\left(l^{1/\lambda}S_1(l)^{-\frac{p}{p-1}}\right)'
\ge
A(l)
\left(
\frac1\lambda
-\frac{p(n-1)}{n(p-1)}
\right),
\]
where \(A(l)>0\). Hence,
\[
\left(l^{1/\lambda}S_1(l)^{-\frac{p}{p-1}}\right)'\ge0
\]
provided that
\[
0<\lambda\le \frac{n(p-1)}{p(n-1)}.
\]
Therefore,
\[
l^{1/\lambda}S_1(l)^{-\frac{p}{p-1}}
\]
is nondecreasing.
\end{proof}
\begin{proof}[Proof of the monotonicity of $F_\kappa(l)S_\kappa(l)^{-\frac{p}{p-1}}$]
Since
\[
S_0(l)^{-\frac{p}{p-1}}
=
C_0\,l^{-\frac{p(n-1)}{n(p-1)}},
\]
and
\[
F_0(l)
=
C_0
\int_0^l
w^{\frac1\lambda-\frac{p(n-1)}{n(p-1)}}
\left(
\int_0^w f^*(s)\,ds
\right)^{\frac1{p-1}}
dw,
\]
it follows that
\[
F_0(l)S_0(l)^{-\frac{p}{p-1}}
=
C_0^2
\left(
\int_0^l
w^{\frac1\lambda-\frac{p(n-1)}{n(p-1)}}
\left(
\int_0^w f^*(s)\,ds
\right)^{\frac1{p-1}}
dw
\right)
l^{-\frac{p(n-1)}{n(p-1)}}.
\]
We compute that
\begin{align*}
&
\left(
\left(
\int_0^l w^{\frac{1}{\lambda}-\frac{p(n-1)}{n(p-1)}}
\left(\int_0^w f^*(s)\,ds\right)^\frac{1}{p-1}dw
\right)
l^{-\frac{p(n-1)}{n(p-1)}}
\right)' \\
&=l^{\frac{1}{\lambda}-\frac{p(n-1)}{n(p-1)}}\left(\int_0^w f^*(s)\,ds\right)^\frac{1}{p-1}
-\frac{p(n-1)}{n(p-1)}l^{-\frac{p(n-1)}{n(p-1)}-1}\\
&\qquad\qquad\qquad\qquad\times\left(
\int_0^l w^{\frac{1}{\lambda}-\frac{p(n-1)}{n(p-1)}}
\left(\int_0^w f^*(s)\,ds\right)^\frac{1}{p-1}dw
\right)\\
&\ge \left(\int_0^w f^*(s)\,ds\right)^\frac{1}{p-1} l^{-1-\frac{2p(n-1)}{n(p-1)}}
\\
&\qquad\qquad\qquad\qquad\times
\left[ l^{\frac{1}{\lambda}+1}- \frac{p(n-1)}{n(p-1)}l^{\frac{p(n-1)}{n(p-1)}}
(\frac{1}{\lambda}-\frac{p(n-1)}{n(p-1)}+1)^{-1}l^{\frac{1}{\lambda}-\frac{p(n-1)}{n(p-1)}+1}
                       \right]\\
&\ge (\frac{1}{\lambda}-\frac{p(n-1)}{n(p-1)}+1)^{-1}\left(\int_0^w f^*(s)\,ds\right)^\frac{1}{p-1} l^{\frac{1}{\lambda}-\frac{2p(n-1)}{n(p-1)}}\left[ \frac{1}{\lambda}-\frac{2p(n-1)}{n(p-1)} +1  \right]
.
\end{align*}
Consequently, the above expression is non-decreasing when
$
0 < \lambda < \frac{n(p-1)}{p(n-2)+n}.
$\par
For $\kappa=1$, by the definitions of $F_1$ and $S_1$, a direct computation yields
\begin{align*}
\left(F_1(l)S_1(l)^{-\frac{p}{p-1}}\right)'
&=
F_1'(l)S_1(l)^{-\frac{p}{p-1}}
-\frac{p}{p-1}
F_1(l)S_1(l)^{-\frac{p}{p-1}-1}S_1'(l)
\nonumber\\
&=
S_1(l)^{-\frac{p}{p-1}-1}
\left(
F_1'(l)S_1(l)
-\frac{p}{p-1}F_1(l)S_1'(l)
\right).
\end{align*}
\noindent
It therefore suffices to study the sign of
\[
F_1'(l)S_1(l)
-\frac{p}{p-1}F_1(l)S_1'(l).
\]
\noindent
Differentiating once more, we obtain
\begin{align*}
&
\left(
F_1'(l)S_1(l)
-\frac{p}{p-1}F_1(l)S_1'(l)
\right)'
\\
&=
F_1''(l)S_1(l)
+F_1'(l)S_1'(l)
-\frac{p}{p-1}F_1'(l)S_1'(l)
-\frac{p}{p-1}F_1(l)S_1''(l)
\\
&=
F_1''(l)S_1(l)
-\frac{1}{p-1}F_1'(l)S_1'(l)
-\frac{p}{p-1}F_1(l)S_1''(l).
\end{align*}
\noindent
Since
\[
S_1''(l)
=
\frac{d}{dl}
\left(
\frac{n-1}{\tan\bigl(L_1^{-1}(l)\bigr)}
\right)
\le 0,
\]
\noindent
Consequently,
\[
\left(
F_1'(l)S_1(l)
-\frac{p}{p-1}F_1(l)S_1'(l)
\right)'
\ge
F_1''(l)S_1(l)
-\frac{1}{p-1}F_1'(l)S_1'(l).
\]
\begin{align*}
F_1''(l)
&=
\frac1\lambda
l^{\frac1\lambda-1}
S_1(l)^{-\frac{p}{p-1}}
\left(
\int_0^l f^*(s)\,ds
\right)^{\frac1{p-1}}
\\
&\quad
-\frac{p}{p-1}
l^{\frac1\lambda}
S_1(l)^{-\frac{p}{p-1}-1}
S_1'(l)
\left(
\int_0^l f^*(s)\,ds
\right)^{\frac1{p-1}}
\\
&\quad
+\frac1{p-1}
l^{\frac1\lambda}
S_1(l)^{-\frac{p}{p-1}}
\left(
\int_0^l f^*(s)\,ds
\right)^{\frac1{p-1}-1}
f^*(l).
\end{align*}
Substituting the above expression for $F_1''$ into the previous inequality, we obtain
\begin{align*}
&
F_1''(l)S_1(l)
-\frac1{p-1}F_1'(l)S_1'(l)
\\
&=
\left(
\int_0^l f^*(s)\,ds
\right)^{\frac1{p-1}}
\Biggl[
\frac1\lambda
l^{\frac1\lambda-1}
S_1(l)^{-\frac{p}{p-1}+1}
\\
&\qquad
-\frac{p}{p-1}
l^{\frac1\lambda}
S_1(l)^{-\frac{p}{p-1}}
S_1'(l)
-\frac1{p-1}
l^{\frac1\lambda}
S_1(l)^{-\frac{p}{p-1}}
S_1'(l)
\Biggr]
\\
&\qquad
+\frac1{p-1}
l^{\frac1\lambda}
S_1(l)^{-\frac{p}{p-1}+1}
\left(
\int_0^l f^*(s)\,ds
\right)^{\frac1{p-1}-1}
f^*(l).
\end{align*}
\noindent
Since $f^*(l)\ge0$, we deduce that
\begin{align*}
&
F_1''(l)S_1(l)
-\frac1{p-1}F_1'(l)S_1'(l)
\\
&\ge
l^{\frac1\lambda-1}
S_1(l)^{-\frac{p}{p-1}}
\left(
\int_0^l f^*(s)\,ds
\right)^{\frac1{p-1}}
\Biggl[
\frac1\lambda S_1(l)
-\frac{p+1}{p-1}lS_1'(l)
\Biggr].
\end{align*}
Returning to \eqref{eq:derivative-FS}, we proceed as in the previous argument. In view of
\[
\frac{\cos r}{\sin^n r}
\int_0^r \sin^{n-1}s\,ds
\le \frac1n,
\]
it follows that
\[
\frac1\lambda
-\frac{p+1}{p-1}(n-1)
\frac{\cos r\int_0^r \sin^{n-1}s\,ds}
{\sin^n r}
\ge
\frac1\lambda
-\frac{(p+1)(n-1)}{n(p-1)}.
\]
\noindent
Therefore,
\[
\frac1\lambda
-\frac{p+1}{p-1}(n-1)
\frac{\cos r\int_0^r \sin^{n-1}s\,ds}
{\sin^n r}
\ge0
\]
whenever
\[
0<\lambda\le
\frac{n(p-1)}{(p+1)(n-1)}.
\]
\noindent
Consequently, the function
\[
l\mapsto
F_1(l)S_1(l)^{-\frac{p}{p-1}}
\]
is nondecreasing on its domain.
\par
For the case $n=2$, we investigate the sign of
\[
F_1'(l)S_1(l)
-\frac{p}{p-1}F_1(l)S_1'(l).
\]
Equivalently, it suffices to study the sign of
\[
F_1'(l)S_1(l)^2
-\frac{p}{p-1}F_1(l)S_1(l)S_1'(l).
\]

We differentiate this quantity with respect to $l$, obtaining
\begin{align*}
&\left(F_1'(l)S_1(l)^2
-\frac{p}{p-1}F_1(l)S_1(l)S_1'(l)\right)' \\
&=
F_1''(l)S_1(l)^2
+2F_1'(l)S_1(l)S_1'(l)
-\frac{p}{p-1}F_1'(l)S_1(l)S_1'(l)
-\frac{p}{p-1}F_1(l)\big(S_1(l)S_1'(l)\big)' \\
&=
F_1''(l)S_1(l)^2
+\frac{p-2}{p-1}F_1'(l)S_1(l)S_1'(l)
-\frac{p}{p-1}F_1(l)\big(S_1(l)S_1'(l)\big)'.
\end{align*}

In dimension $n=2$, we can directly verify that
\[
\big(S_1(l)S_1'(l)\big)' = A(\cos r)' \le 0,
\]
where $A>0$ is a constant.
It follows that
\begin{align*}
\left(F_1'(l)S_1(l)^2
-\frac{p}{p-1}F_1(l)S_1(l)S_1'(l)\right)'
\ge
F_1''(l)S_1(l)^2
+\frac{p-2}{p-1}F_1'(l)S_1(l)S_1'(l).
\end{align*}

We proceed as in the previous case and obtain
\begin{align*}
F_1''(l)S_1(l)^2
+\frac{p-2}{p-1}F_1'(l)S_1(l)S_1'(l)
\ge
S_1(l)^{-\frac{p}{p-1}} l^{\frac{1}{\lambda}-1}
\left[
\frac{1}{\lambda}S_1(l)
-\frac{2}{p-1}lS_1'(l)
\right].
\end{align*}

We now return to the previous argument. Consequently, we conclude that
\[
F_1(l)S_1(l)^{\frac{p}{p-1}}
\ \text{is nondecreasing for } 0<\lambda\le p-1.
\]
\end{proof}
\begin{remark}
One may expect that more refined estimates could lead to improved inequalities.
\end{remark}
\bibliographystyle{plain}
\bibliography{graduated.bib}
\end{document}